\documentclass[11pt]{article}
\usepackage{amsfonts}
\usepackage{graphicx}
\usepackage{amsmath}
\usepackage{amsthm}
\usepackage{amssymb}
\usepackage{amscd}
\usepackage{color}

\usepackage{graphicx}
\usepackage{microtype}
\usepackage{enumerate}
\usepackage{fullpage}
\usepackage{pinlabel}
\usepackage[top=1.2in,bottom=1.6in,left=1.2in,right=1.2in]{geometry}

\theoremstyle{definition}
\newtheorem{theorem}{Theorem}[section]

\newtheorem{claim}[theorem]{Claim}
\newtheorem{conjecture}[theorem]{Conjecture}
\newtheorem{corollary}[theorem]{Corollary}

\newtheorem{lemma}[theorem]{Lemma}
\newtheorem{notation}[theorem]{Notation}

\newtheorem{proposition}[theorem]{Proposition}
\newtheorem{remark}[theorem]{Remark}
\newtheorem{question}[theorem]{Question}
\newtheorem{definition}[theorem]{Definition}
\newtheorem{no/co}[theorem]{Notation/Conventions}

\theoremstyle{plain}
\newtheorem{condition}[theorem]{}

\newcommand{\R}{\mathbb{R}}
\newcommand{\Z}{\mathbb{Z}}
\newcommand{\N}{\mathbb{N}}
\newcommand{\Q}{\mathbb{Q}}
\newcommand{\C}{\mathbb{C}}
\newcommand{\B}{\mathbb{B}}
\newcommand{\del}{\partial}

\DeclareMathOperator{\id}{id}

\DeclareMathOperator{\supp}{supp}
\DeclareMathOperator{\vol}{vol}
\DeclareMathOperator{\Diff}{Diff}
\DeclareMathOperator{\Homeo}{Homeo}
\DeclareMathOperator{\rel}{rel}
\DeclareMathOperator{\GL}{GL}
\DeclareMathOperator{\germ}{\mathcal{G}_+}

\title{Automatic continuity for homeomorphism groups and applications}
\author{Kathryn Mann \\ Appendix with Fr\'ed\'eric Le Roux}
\date{}

\begin{document}

\maketitle

\abstract{Let $M$ be a compact manifold, possibly with boundary.  We show that the group of homeomorphisms of $M$ has the \emph{automatic continuity property}: any homomorphism from $\Homeo(M)$ to any separable group is necessarily continuous.  This answers a question of C. Rosendal.  If $N \subset M$ is a submanifold, the group of homeomorphisms of $M$ that preserve $N$ also has this property.   

Various applications of automatic continuity are discussed, including applications to the topology and structure of \emph{groups of germs} of homeomorphisms. 
In an appendix with Fr\'ed\'eric Le Roux we also show, using related techniques, that the group of germs at a point of homeomorphisms of $\R^n$ is strongly uniformly simple.  
}

\section{Introduction}

\begin{definition} A topological group $G$ has the \emph{automatic continuity property} if every homomorphism from $G$ to any separable group $H$ is necessarily continuous.   
\end{definition}

One should think of automatic continuity as a very strong form of rigidity. Many familiar topological groups fail to have the property, for example
\begin{itemize}
\item  Automorphisms of $\R$ as a vector space over $\Q$ (other than homotheties) are discontinuous homomorphisms $\R \to \R$.  
\item Applying a wild automorphism of $\C$ to all matrix entries gives a discontinuous homomorphism $\GL(n, \C) \to \GL(n, \C)$. 
\item More generally, for any field $F$ of cardinality at most continuum, Kallman \cite{Kallman} gives injective homomorphisms from $\GL(n, F)$ to $S_\infty$, the group of permutations of an infinite countable set.  As $S_\infty$ admits a separable, totally disconnected topology, $\GL(n, F)$ fails to have automatic continuity as soon as $F$ is not totally disconnected.  
\end{itemize}

Remarkably, several ``big" groups do have the automatic continuity property; a current research program in descriptive set theory aims to show that automorphism groups of certain structures have automatic continuity.  Examples of such groups known to have automatic continuity include the infinite-dimensional unitary group \cite{Tsankov}, the group of isometries of the Urysohn space \cite{Sabok}, the order-preserving automorphisms of $\Q$ and homeomorphisms of $2^\N$ and of $\R$ \cite{RS}, and the homeomorphism groups of compact 2-dimensional manifolds \cite{Rosendal 2man}.  The primary goal of this paper is to prove the following.  

\begin{theorem} \label{main cont thm}
Let $M$ be a compact manifold, possibly with boundary, and let $\Homeo_0(M)$ denote the identity component of the group of homeomorphisms of $M$ (with the standard $C^0$ topology).   Then $\Homeo_0(M)$ has automatic continuity.  
\end{theorem}

Of course, this immediately implies that $\Homeo(M)$ has automatic continuity as well.  We also prove automatic continuity for the subgroup of homeomorphisms of $M$ that preserve a submanifold, and a form of automatic continuity for homeomorphism groups of noncompact manifolds.

\subsection{Main applications}

The proof of Theorem \ref{main cont thm} indicates that there is a deep relationship between the topology of a manifold and the topology and algebraic structure of its homeomorphism group.  In Section \ref{applications sec} we describe three applications.  \medskip

\noindent \textit{I. Uniqueness results.}
A first consequence of automatic continuity is a new proof of a theorem of Kallman.

\begin{theorem}[\cite{Kallman unique}] \label{Kallman thm}
Let $M$ be a compact manifold. The group $\Homeo_0(M)$ has a unique complete, separable topology. 
\end{theorem}

\medskip
\noindent \textit{II. Extension problems.}   Epstein and Markovic \cite{EM} asked whether every \emph{extension homomorphism} $\Homeo_0(S^1) \to \Homeo_0(D^2)$ is continuous.  Automatic continuity immediately gives a positive answer, as well as for the more general question of extensions replacing the pair $(D^2, S^1)$ with $(M, N)$, where $N$ is either a submanifold or boundary component of $M$.   A more subtle variant of this question was asked by Navas.  

\begin{question}[Navas] 
Let $\germ(\R^n, 0)$ denote the group of germs at $0$ of orientation preserving homeomorphisms of $\R^n$ fixing 0.  Does there exist an extension homomorphism $\germ(\R^n, 0) \to \Homeo(\R^n, 0)$?  
\end{question}

We use automatic continuity to prove a much stronger result, which also implies that the group of germs does not admit a separable topology.

\begin{theorem} 
Let $H$ be any separable group.  Then any homomorphism $\germ(\R^n, 0) \to H$ is trivial.  
\end{theorem}

\noindent The proof uses the algebraic simplicity of $\germ(\R^n, 0)$, a strong form of which is proved in the appendix.  

We also discuss related problems on homomorphisms between groups of germs, and progress on problems involving homomorphisms between groups of homeomorphisms, a topic that has recently attracted significant attention.

\medskip
\noindent \textit{III. Nonsmoothing.} A third application is a global ``algebraic nonsmoothing" theorem.   Recall that an action of a group $G$ on a manifold $M$ is $C^r$-\emph{smoothable} if it is topologically conjugate to an action by $C^r$ diffeomorphisms.  
Although this is a dynamical constraint on the action, it is also interesting to ask whether the \emph{algebraic structure} of $G$ is an obstruction to actions of higher regularity.

Motivated by this, define the \emph{regularity} of an abstract group $G$ to be the largest $r$ such that there exists a manifold $M$ and a nontrivial homomorphism $G \to \Diff^r(M)$.   If $G \subset \Homeo_0(M)$, we call an action of $G$ on a manifold $N$ by $C^r$ diffeomorphisms a \emph{algebraic $C^r$--smoothing} of $G$.    We show the following.  

\begin{theorem}  \label{nonsmoothing thm} 
Let $M$ be a compact manifold.  Then $\Homeo_0(M)$ has regularity 0.  In particular, $\Homeo_0(M)$ is not algebraically $C^1$--smoothable.  
\end{theorem}

\bigskip
\noindent \textbf{Acknowledgements.} The author thanks Ian Agol, Charles Pugh, and Franco Vargas Pallete for their interest in this project, Benson Farb and Bena Tshishiku for their comments on early versions of this manuscript, and Christian Rosendal for his ongoing interest and support.   
This work was partially completed while the author was in residence at MSRI, supported by NSF grant 0932078

\section{The structure of $\Homeo(M)$} \label{background sec}

We introduce some algebraic and topological properties of homeomorphism groups that will be used throughout the paper.   Much of the material in this section is standard. 

\begin{definition} 
Let $M$ be a compact manifold.  The $C^0$ topology on $\Homeo(M)$ is induced by the metric
$$d(f, g) := \max_{x \in M} d_M\left( f(x), g(x) \right)$$
where $d_M$ is any compatible metric on $M$.  
\end{definition}
\noindent This topology is separable, and independent of the choice of metric on $M$, provided that $d_M$ is compatible, i.e. it generates the topology of $M$.  Although the metric given above is not complete, $\Homeo(M)$ does admit a complete metric, in fact the metric $D(f, g) :=  d(f, g) + d(f^{-1}, g^{-1})$ is such an example (c.f. Corollary 1.2.2 in \cite{BK}).

If $N \subset M$ is a closed $d$-dimensional submanifold -- meaning that the pair $(M, N)$ has local charts to $(\R^n, \R^d)$ -- define the \emph{relative homeomorphism group}  
$$\Homeo(M \rel N) := \{ f \in \Homeo(M) \mid f(N) = N \}$$
This is a $C^0$--closed subgroup and hence also completely metrizable. Its identity component is denoted by $\Homeo_0(M \rel N)$. 

\bigskip
\noindent \textbf{Support and fragmentation.} The \emph{support} of a homeomorphism $f$, denoted $\supp(f)$, is the closure of the set $\{x \in M \mid f(x) \neq x\}$.
In our proof of automatic continuity, we will make frequent use of the fact that homeomorphisms with small support are close to the identity -- this is the most basic relationship between the topology of $\Homeo_0(M)$ and $M$.  

Although homeomorphisms close to the identity need not have small support, the \emph{fragmentation property} states that a homeomorphism sufficiently close to the identity can be expressed as the product of a \emph{bounded number} of homeomorphisms with small support.  

\begin{definition}[Local fragmentation] \label{fragmentation}
A group $G \subset \Homeo(M)$ has the \emph{local fragmentation property} if the following holds.  
Given any finite open cover $\{E_1, ... E_m\}$ of $M$, there exists a neighborhood $U$ of the identity in $G$ such that each $g \in U$ can be factored as a composition 
$g = g_1 g_2  \ldots  g_m$,
where $\supp(g_i) \subset E_i$.  
\end{definition}

\begin{proposition}[Edwards--Kirby, \cite{EK}]  \label{fragmentation prop}
Let $M$ be a compact manifold, possibly with boundary.  Then $\Homeo_0(M)$ has the local fragmentation property.   If $N \subset M$ is an embedded submanifold, then $\Homeo_0(M \rel N)$ also has the local fragmentation property.  
\end{proposition}
 
\begin{proof} 
The proof for $\Homeo_0(M)$ is given in Corollary 1.3 of Edwards--Kirby \cite{EK}, it uses the topological torus trick.   The case of $\Homeo_0(M \rel N)$ also follows from Edwards and Kirby's work by using the relative version of their deformation theorem, this is explained in Remark 7.2 of \cite{EK}.
\end{proof} 

\bigskip
\noindent \textbf{Perfectness.}  Recall that a group $G$ is \emph{perfect} if any element can be written as a product of commutators.  We will use the following uniform version for homeomorphisms supported on embedded balls.    

\begin{proposition}[Uniform perfectness] \label{perfectness prop}
Let $B \subset M$ be an embedded open ball.  Then any $f \in \Homeo_0(M)$ with $\supp(f) \subset B$ can be written as $f = [a,b]$ where $\supp(a) \subset B$ and $\supp(b) \subset B$.
\end{proposition}

\begin{proof} This result is ``folklore", the earliest proof known to the author is the following argument of Anderson \cite{Anderson}.  Suppose that $\supp(f) \subset B$.  Since $\supp(f)$ is compact and $B$ open, there exists $b \in \Homeo_0(M)$ with $\supp(b) \subset B$ and such that $b^n(\supp(f)) \cap b^m(\supp(f)) = \emptyset$ for any $m \neq n$.   Define $a$ by 
$$ a(x) = \left\{ \begin{array}{ll} b^{n} f b^{-n}(x) & \text{ if } x \in b^n(\supp(f)), \text{ for some } n  \\
x & \text{ otherwise} 
 \end{array} \right.$$
Then $\supp(a) \subset B$ and $[a, b] = f$. 
\end{proof} 

Anderson's argument can easily be modified to give a relative version for balls intersecting $\del M$ or intersecting an embedded submanifold $N$.  
To state this precisely, define a \emph{half ball} in $M^n$ to be a proper embedding of $\{(x_1, ... x_n) \in \B^n \mid x_n \geq 0\}$, i.e. with the image of $\B^n \cap \{x_n = 0\}$ in $\del M$; and similarly if $N^d \subset M^n$ is an embedded $d$-dimensional submanifold, define a \emph{relative ball} to be an embedding $\psi: \B^n \to M$ such that $\psi( \B^n \cap \R^d) = N^d \cap \psi(\B^n)$.  

\begin{proposition} \label{relative perfectness prop} (Uniform perfectness, relative case)
\begin{enumerate}[i)]
\item Let $M$ be a manifold of dimension at least $2$, and $B$ an open half-ball intersecting $\del M$.  Any $f \in \Homeo_0(M)$ with $\supp(f) \subset B$ can be written as $f = [a,b]$ where $a$ and $b$ are both supported in $B$.
\item Let $N \subset M$ be an embedded submanifold of dimension at least 1, and let $B$ be a relative ball in $M$.  Then any $f \in \Homeo_0(M \rel N)$ with $\supp(f) \subset B$ can be written as $f = [a,b]$ where $a, b \in \Homeo_0(M \rel N)$ are both supported in $B$.
\end{enumerate}
\end{proposition}
The proof is exactly the same -- it requires no changes if $\supp(f)$ is disjoint from $\del M$ or $N$, and if $\supp(f) \cap \del M \neq \emptyset$ or $\supp(f) \cap N \neq \emptyset$ one simply takes the iterates of $b$ to translate $\supp(f)$ to a collection of disjoint relative or half-balls in $B$.   

\begin{remark}[Perfectness of homeomorphism groups].
The reader will note that Proposition \ref{perfectness prop} together with fragmentation implies that the groups $\Homeo_0(M)$ and $\Homeo_0(M \rel N)$ are perfect, though not necessarily \emph{uniformly} perfect -- a much more subtle question. \end{remark}


\section{Proof of Theorem \ref{main cont thm}} \label{main pf sec}

For the rest of this section, we fix a separable topological group $H$, and assume that $\phi: \Homeo_0(M) \to H$ is a homomorphism.   
For simplicity, we first treat only the case where $M$ is closed; modifications for the case where $\del M \neq \emptyset$ and the relative case of $\Homeo_0(M \rel N)$ are discussed in Section \ref{modification sec}, along with a comment for noncompact $M$.

The proof is somewhat involved, so we have divided it into three major steps.  The first is general set-up; the second a ``localized" version of continuity (for homeomorphisms with support in a small ball), and the third step improves this local result to a global version by careful use of the fragmentation property.  There is a  delicate balancing act between steps 2 and 3; in particular, it will be necessary to construct a particular kind of \emph{efficient} cover of the manifold $M$ to use in the fragmentation argument.

\subsection*{Step 1: set-up for the proof}

Since $\phi$ is a group homomorphism, it suffices to show continuity at the identity.  In other words, we need to prove the following. 
\begin{condition} \label{main condition}
For any neighborhood $V$ of the identity in $H$, there exists a neighborhood $U$ of the identity in $\Homeo_0(M)$ such that $U \subset \phi^{-1}(V)$.  
\end{condition}

\noindent We first use the Baire category theorem to extract an ``approximate" version of condition \ref{main condition}.   

\begin{lemma}  \label{Baire lemma}
Let $V$ be a neighborhood of the identity in $H$.  There exists a neighborhood $U$ of the identity in $\Homeo_0(M)$ such that $U$ is contained in the \emph{closure} of $\phi^{-1}(V)$. 
\end{lemma}

\begin{remark} The proof given below works generally for homomorphisms from any Polish group to a separable group; see e.g. \cite{EM} for one instance of this (where it is called ``Baire category continuity") and \cite{RS} for another in the context of groups with the ``Steinhaus property". 
\end{remark}

\begin{proof}[Proof of Lemma \ref{Baire lemma}]
Take a smaller neighborhood of the identity $V_0 \subset V$ such that $V_0$ is symmetric (i.e. $v \in V_0 \Leftrightarrow v^{-1} \in V_0$), and such that $V_0^4 \subset V$.    Let $\{h_i \}$ be a countable dense subset of $H$, so that 
$$H = \bigcup_i h_i V_0.$$
Let $W = \phi^{-1}(V_0^2)$.  
For each translate $h_i V_0$ that intersects the image of $\Homeo_0(M)$, choose an element $\phi(g_i) \in h_i V_0$.  Then $\phi(g_i) = h_i v_i$ for some $v_i \in V$, and so 
$$h_i V_0 = \phi(g_i) v_i^{-1} V_0 \subset \phi(g_i) V_0^2.$$
Thus 
$$H = \bigcup_i \phi(g_i) V_0^2$$
and, taking pre-images, we have 
$$\Homeo_0(M) = \bigcup_i g_i W.$$
Since $\Homeo_0(M)$ is a Baire space, it cannot be covered by countably many nowhere dense sets.  Thus, $W$ is dense in the neighborhood of some $g \in \Homeo_0(M)$, so 
$$WW^{-1} = \phi^{-1}(V_0^4) \subset \phi^{-1}(V)$$ 
is dense in some neighborhood of the identity in $\Homeo_0(M)$.   This proves the lemma.  
\end{proof}

\noindent Of course, improving ``dense in a neighborhood of the identity" to ``contains a neighborhood of the identity" is a nontrivial matter and the main goal of this work!

\subsection*{Step 2: A localized version (after Rosendal)} 

As in Step 1, assume that we have fixed a homomorphism $\phi: \Homeo_0(M) \to H$, and a neighborhood $V$ of the identity in $H$, with the aim of showing that $\phi$ satisfies condition \ref{main condition}.  

The ``localized version" of Condition \ref{main condition} that we aim to prove here states, loosely speaking, that a homeomorphism $f \in \Homeo_0(M)$ with sufficiently small support lies in $\phi^{-1}(V)$.    The precise statement that we will use in the next step is given in Lemma \ref{multi lemma} below.   Our strategy is to build up to this statement gradually, using a series of lemmata guided by Rosendal's work in \cite{Rosendal 2man}.   

\begin{notation} \label{W notation}
As in Lemma \ref{Baire lemma}, we start by fixing a smaller, symmetric neighborhood of the identity $V_0 \subset V$ such that $V_0^8 \subset V$.  Let $W = \phi^{-1}(V_0)$.  
\end{notation}  

The first lemma is a very rough version of our end goal.  It states that in any neighborhood of any point of $M$, we can find an open ball so that all diffeomorphisms supported on that ball are \emph{restrictions} of elements in $\phi^{-1}(V_0^2) \subset \phi^{-1}(V)$.    

\begin{lemma}  \label{restriction lemma} 
Let $B \subset M$ be an embedded ball.  
There exists a ball $B' \subset B$ such that for every $f \in \Homeo_0(M)$ with $\supp(f) \subset B'$, there is an element $w _f \in W^2$ with $\supp(w) \subset B$ and such that the restriction of $w_f$ to $B'$ agrees with $f$.  
\end{lemma} 

\begin{proof} 
Let $B \subset M$ be an embedded ball.  
The argument from the proof of Lemma \ref{Baire lemma} implies that there exists a countable set $\{g_i\} \subset \Homeo_0(M)$ such that 
$$\Homeo_0(M) = \bigcup_i g_i W.$$
We first prove a related claim for these translates of $W$.  

\begin{claim}  \label{kiW claim}
There exists a ball $B' \subset B$, and a left translate $g_i W$ such that if $\supp(f) \subset B'$, then there exists $w_f \in g_i W$ such that 
\begin{enumerate}[i)]
\item $\supp(w_f) \subset B$, and 
\item the restriction of $w_f$ to $B'$ agrees with $f$.
\end{enumerate}
\end{claim}

\noindent \textit{Proof of claim}.  Let $B_i$, $i=1, 2,...$ be a sequence of disjoint balls with disjoint closures and with the closure of $\bigcup \limits_{i=1}^\infty B_i$ contained in $B$.  

We will show that for some $i$, every $f \in \Homeo_0(M)$ with $\supp(f) \subset B_i$ agrees with the restriction of an element of $g_i W$ supported on $B$.   

Suppose for contradiction that this is not the case.  Then there is a sequence $f_i \in \Homeo_0(M)$ with $\supp(f_i) \subset B_i$ and such that $f_i$ does not agree with the restriction to $B_i$ of any element of $g_i W$ supported on $B$.  Using this sequence of counterexamples, define a homeomorphism $F(x)$ by 
$$F(x) = \left\{ \begin{array}{ll} f_i(x) & \text{ if } x \in B_i \text{ for some } i  \\
x & \text{ otherwise} 
 \end{array} \right.$$
Since the translates of $W$ cover $\Homeo_0(M)$, there is some $g_i$ such that $F \in g_i W$.  But by construction $F$ restricts to $f_i$ on $B_i$ -- this gives a contradiction and proves the claim.  \qed

To finish the proof of Lemma \ref{restriction lemma}, let $B' \subset B$ be the ball given by Claim \ref{kiW claim}.  Let $f \in \Homeo_0(M)$ satisfy $\supp(f) \subset B'$.  Then $f$ is the restriction to $W$ of some $w_1 \in g_iW$.   Since $\id \in \Homeo_0(M)$ has trivial support, Claim \ref{kiW claim} implies that there exists some $w_2 \in g_iW$ restricting to the identity on $B'$, so define
$$w_f :=  w_2^{-1}w_1 \in W g_i^{-1} g_i W = W^2.$$
Then $\supp(w_f) \subset B$ and $w_f$ restricted to $B'$ agrees with $f$.  

\end{proof} 

Lemma \ref{restriction lemma} states that certain homeomorphisms with small support are \emph{restrictions} of elements of $W^2 \subset \phi^{-1}(V)$; recall that our goal is to show that homeomorphisms with small support \emph{are} elements of $\phi^{-1}(V)$.  We remove the ``restriction" condition now, at the cost of enlarging $W^2$ to $W^8$, by using a trick with commutators.   Since $W^8 \subset \phi^{-1}(V)$, this will achieve our goal.  

\begin{lemma}  \label{local lemma} 
Let $B \subset M$ be an embedded ball.  There exists a ball $B'' \subset B$ such that, if $f \in G$ has $\supp(f) \subset B''$, then $f \in W^8$.  
\end{lemma}

\begin{proof}[Proof of Lemma \ref{local lemma}]
Let $B \subset M$ be an embedded ball.  Apply Lemma \ref{restriction lemma} to find a ball $B' \subset B$ such that if $\supp(f) \subset B'$, then $f$ agrees on $B'$ with the restriction of an element $w_f \in W^2$ with $\supp(w_f) \subset B$.  

Now apply Lemma \ref{restriction lemma} to $B'$ to find a smaller ball $B'' \subset B' \subset B$ such that if $f \in \Homeo_0(M)$ has $\supp(f) \subset B''$, then $f$ agrees on $B''$ with the restriction of an element $w_f \in W^2$ with $\supp(w_f) \subset B'$.

Let $f$ have support in $B''$. Using Proposition \ref{perfectness prop}, write $f=[a,b]$, where $\supp(a) \subset B''$ and $\supp(b) \subset B''$.
By Lemma \ref{restriction lemma}, there exists $w_a \in W^2$ with $\supp(w_a) \subset B'$ and such that the restriction of $w_a$ to $B''$ agrees with $a$.  There also exist $w_b \in W^2$ with $\supp(w_b) \subset B$ and such that the restriction of $w_b$ to $B'$ agrees with $b$.  Since $\supp(w_a) \cap \supp(w_b) \subset B''$, we have $[a, b] = [w_a, w_b]$ and hence 
$$f = [a, b] = [w_a, w_b] \in W^8.$$
\end{proof}

\noindent Summarizing our work so far, we shown the following. 
\begin{quote}
\textit{For any embedded ball $B \subset M$, there exists a ball $B' \subset B$ such that, if $\supp(f) \subset B'$ then $f \in \phi^{-1}(V)$.}
\end{quote}

In other words, homeomorphisms that are supported in $B'$ have image close to the identity under $\phi$.   At this point, the natural (naive) strategy to finish the proof would be to try and use fragmentation to write any homeomorphism close to the identity as a bounded product of homeomorphisms supported on balls ``like" $B'$.  Unfortunately, Lemma \ref{local lemma} gives us no control on the size of $B'$, so we do not yet have any means to reasonably cover $M$ with a bounded number of balls that have this property. 

To remedy this, we first strengthen Lemma \ref{local lemma} to a similar statement for disjoint unions of balls, getting us closer to a genuine cover of $M$.  

\begin{lemma}  \label{multi lemma}
Let $\{ B_\alpha \}$ be a finite collection of disjoint open balls in $M$.  Then there exist open balls $B''_\alpha \subset B_\alpha$ such that, if $f \in \Homeo_0(M)$ has $\supp(f) \subset B''_\alpha$, then $f \in  \phi^{-1}(V).$
\end{lemma}

The proof consists in running the arguments from Lemmas \ref{restriction lemma} to \ref{local lemma} on all the balls $B_\alpha$ simultaneously.  We state below the necessary modifications to do this.  

\begin{proof}
Let $W$ be as in Notation \ref{W notation} above.  
First, we modify Claim \ref{kiW claim} as follows.
\begin{claim} 
There exist balls $B'_\alpha \subset B_\alpha$, and a left translate $g_i W$ such that, if $\supp(f) \subset \sqcup_\alpha B'_\alpha$, then there exists $w_f \in g_i W$ with $\supp(w_f) \subset \sqcup_\alpha B_\alpha$ and such that the restriction of $w_f$ to $\sqcup_\alpha B'_\alpha$ agrees with $f$.
\end{claim}
To prove this claim, imitate the proof of Claim \ref{kiW claim} by taking a sequence of disjoint balls $B_{\alpha, i} \subset B_\alpha$ for each $\alpha$, and supposing for contradiction that there existed $f_i$ supported on $\sqcup_\alpha B_{\alpha, i}$ but not in $g_iW$.  Taking the ``infinite composition" $F$ of the $f_i$ as before gives the desired contradiction.  
Composing with the identity as in Lemma \ref{restriction lemma} now shows that any $f \in \Homeo_0(M)$ with support in $\sqcup_\alpha B'_\alpha$ actually lies in $W^2$.  

Finally, as in the original proof of Lemma \ref{local lemma}, we can apply this construction twice to find balls $B''_\alpha \subset B'_\alpha \subset B_\alpha$ such that if $\supp(a) \subset \sqcup_\alpha B''_\alpha$, then $f$ agrees with the restriction of an element of $W^2$ supported on $\sqcup_\alpha B'_\alpha$, and if $\supp(b) \subset \sqcup_\alpha B'_\alpha$, then $b$ agrees with an element of $W^2$ supported on $\sqcup_\alpha B_\alpha$.  The same commutator trick using Proposition \ref{perfectness prop} now applies to show that any element $f \in \Homeo_0(M)$ with $\supp(f) \subset \sqcup_\alpha B''_\alpha$ lies in $W^8$.

\end{proof}

\subsection*{Step 3: Local to global}
To finish the proof, we will improve the local result of Lemma \ref{multi lemma} to a global result by using fragmentation with respect to an \emph{efficient cover}, in the sense described below.  We assume that $M$ has been given a metric $d_M$.  

\begin{lemma}[Existence of an efficient cover] \label{cover lemma}
Let $M$ be a compact manifold.  There exists $m \in \N$ (depending only on $M$) such that, for all $\epsilon$ sufficiently small, there is a cover $\{E_1, E_2, ... E_m\}$ of $M$ where each set $E_i$ consists of a finite union of disjoint balls of radius $\epsilon$.  
\end{lemma} 

The critical point of this lemma is that the constant $m$ does not depend on $\epsilon$.  
We defer the proof of the lemma to the end of this section, showing first how to use the efficient cover to finish the proof of the main theorem.  

\begin{proof}[Proof of Theorem \ref{main cont thm}, given Lemma \ref{cover lemma}]
As before, let $\phi: \Homeo_0(M) \to H$ be a homomorphism, let $V \subset H$ be a neighborhood of the identity, and let  $m$ be the constant given by Lemma \ref{cover lemma}.  
Using our usual trick, we take a smaller symmetric neighborhood of the identity $V_0$ such that $V_0^{12m} \subset V$.  

Let $W = \phi^{-1}(V_0)$.  In order to show that $\phi$ satisfies condition \ref{main condition}, it suffices to find a neighborhood $U$ of the  identity in $\Homeo_0(M)$ such that 
$$U \subset W^{12m} = \phi^{-1}(V).$$

By Lemma \ref{Baire lemma}, there exists a neighborhood of the identity in $\Homeo_0(M)$ contained in the closure of $W^2$.  Let $\epsilon$ be small enough so that this neighborhood contains the set 
$$\{ f \in \Homeo_0(M) \mid d_M(f(x), x) < 2 \epsilon \text{ for all } x \in M \}.$$
In particular, any homeomorphism supported on a ball of radius $\epsilon$ in $M$ is contained in this neighborhood, hence in the closure of $W^2$.   Now using this $\epsilon$, build an efficient cover $\{E_1, E_2, ... E_m\}$ as in Lemma \ref{cover lemma}, and let $\{ B^i_\alpha \}$ denote the set of disjoint balls of radius $\epsilon$ comprising $E_i$.   

By Lemma \ref{multi lemma}, there exist balls $(B^i_\alpha)'' \subset B^i_\alpha$ such that any $f$ with $\supp(f) \subset \sqcup_\alpha (B^i_\alpha)''$ satisfies $f \in W^8$.  
Let $h_i \in \Homeo_0(M)$ be supported on a small neighborhood of $\sqcup_\alpha B^i_\alpha$, and such that  the closure of $h_i(B^i_\alpha)$ is contained in $(B^i_\alpha)''$. Since the balls $B^i_\alpha$ have radius $\epsilon$, there exists such a homeomorphism $h_i$ with $\sup_{x \in M} d_M (h_i(x), x) < 2 \epsilon$.  Thus, $h_i$ can be approximated by some element $w_i \in W^2$.  In particular this implies that there exists $w_i \in W^2$ such that $w_i(\sqcup_\alpha B^i_\alpha) \subset (\sqcup_\alpha B^i_\alpha)''$.  

We can now finally show that there is a neighborhood of the identity in $\Homeo_0(M)$ contained in $W^{12m}$.  
Since $\Homeo_0(M)$ has the local fragmentation property (Proposition \ref{fragmentation prop}), there is a neighborhood $U$ of the identity in $\Homeo_0(M)$ such that for all $f \in U$, we can write 
$$f = f_1  f_2  ... f_m$$ 
with $\supp(f_i) \subset E_i$.  Then $\supp(w_i^{-1} f_i w_i) \subset \sqcup_\alpha (B^i_\alpha)''$, so $w_i^{-1} f_i w_i \in W^8$ which implies that $f_i \in W^{12}$.  It follows that $f \in W^{12m}$, as desired.    This completes the proof of the theorem.  

\end{proof}

It remains only to prove Lemma \ref{cover lemma}.   Our construction uses \emph{$\epsilon$-nets}, so we begin by recalling the definition of an $\epsilon$-net.  Here, and in the proof of the Lemma, we use the notation $B(r, x)$ for the ball of radius $r$ about a point $x$.  

\begin{definition}
Let $M$ be a manifold with metric.  An \emph{$\epsilon$-net} is a finite set of points $\{x_1, ... x_k\}$ in $M$ satisfying 
\begin{enumerate}[i)]
\item $\bigcup \limits_{i=1}^n = B(\epsilon, x_i) = M$, and
\item $B(\epsilon/2, x_i) \cap B(\epsilon/2, x_j) = \emptyset$ for $i \neq j$.
\end{enumerate}
Similarly, we define an \emph{$\epsilon$-net cover} of $M$ is a set of balls $\{B(\epsilon, x_i)\}$ such that the union of the centers $\{x_i\}$ forms an $\epsilon$-net.  
\end{definition}  

The content of the proof is to show that, for any $\epsilon$ small enough, an $\epsilon$-net cover will satisfy the requirements of the Lemma.  

\begin{proof}[Proof of Lemma \ref{cover lemma}]
Let $M$ be a compact manifold.  To simplify the proof, we assume that $M$ is equipped with a Riemannian metric, although it is possible to give a purely combinatorial argument using an arbitrary compatible metric on $M$.     

Recall that, for any cover $\{A_i\}$, the \emph{dual graph} of the cover is the graph with vertex set $\{A_i\}$ and an edge between $A_i$ and $A_j$ whenever $A_i \cap A_j \neq \emptyset$.    
We first prove a regularity result on the dual graphs of $\epsilon$-net covers.  

\begin{claim}
There exists $\lambda > 0$ and $m = m(M) \in \N$ such that, for every $\epsilon < \lambda$, the dual graph of any $\epsilon$-net cover of $M$ has degree less than $m$. 
\end{claim}

\begin{proof}
Take $\epsilon > 0$, and let $\{B(\epsilon, x_i)\}$ be an $\epsilon$-net cover.  
If there is an edge between $\{B(\epsilon, x_i)\}$ and $\{B(\epsilon, x_j)\}$ in the dual graph, then $x_j \in B(2\epsilon, x_i)$.  Since $B(\epsilon/2, x_i) \cap B(\epsilon/2, x_j) = \emptyset$, the degree of the vertex $\{B(\epsilon, x_i)\}$ is bounded above by 
$$\inf \left\{ \frac{\vol(B(2\epsilon, x_i))}{\vol(B(\epsilon/2, x))} \mid x \in B(2\epsilon, x_i)  \right\} $$
Since $M$ is compact, the limit of this ratio as $\epsilon \to 0$ is bounded, and can be taken independent of the point $x_i$. Let $m$ be any integer so that $m-1$ is strictly larger than this bound.  
It follows that, if $\lambda > 0$ is sufficiently small, then for any $\epsilon < \lambda$, the degree of a vertex in the dual graph to any $\epsilon$-net cover will be bounded by $m-1$.  
\end{proof} 

To finish the proof of the lemma, note that any graph of degree less than $m$ admits a proper coloring with $m$ colors.  (In fact, this bound is given by the greedy coloring).  As a consequence, if $m$ and $\lambda$ are the constants from the claim, then for all $\epsilon < \lambda$, any $\epsilon$-net cover of $M$ can be partitioned into $m$ subsets $E_1, E_2, ... E_m$, each consisting of a disjoint union of balls of radius $\epsilon$.

\end{proof}

\begin{remark}
With a little more work (e.g. covering first a neighborhood of the 1-skeleton of a triangulation, then 2-cells, etc.) it should be possible to produce a constant $m$ that depends only on the dimension of $M$.  But we do not need this stronger fact.
\end{remark}

\section{A broader picture}

\subsection{Automatic continuity in the relative and boundary case} \label{modification sec}

To prove automatic continuity for $\Homeo_0(M)$ when $\del M \neq \emptyset$, or for $\Homeo_0(M \rel N)$ when $N$ is an embedded submanifold of dimension at least 1, one needs essentially no new ingredients besides the relative versions of perfectness and fragmentation stated in Section \ref{background sec}.   Step 1 of the proof carries through verbatim, we list here the necessary modifications in Step 2 and 3.  

\bigskip
\noindent \textbf{Step 2: the ``local" version for half- or relative-balls. }
Recall that a \emph{half ball} in $M$ is a proper embedding of $\{(x_1, ... x_n) \in \B^n \mid x_n \geq 0\}$.  Lemma \ref{local lemma} has a straightforward reformulation for half-balls.  

\begin{lemma}
Let $B \subset M$ be an embedded half-ball.  Then there exists a half-ball $B'' \subset B$ such that, if $f \in G$ has $\supp(f) \subset B''$, then $f \in W^8$.  
\end{lemma}

The proof is identical to the proof for balls in the interior of $M$, one simply replaces ``ball" with ``half-ball" everywhere, starting in Lemma \ref{restriction lemma}, and uses the version of Proposition \ref{perfectness prop} for half-balls.   This gives a version of Lemma \ref{multi lemma} for manifolds with boundary.  

\begin{lemma}  \label{multi lemma 2}
Let $\{ B_\alpha \}$ be a finite collection of disjoint open balls or half-balls in $M$.  Then there exist open balls or, respectively, half-balls $B''_\alpha \subset B_\alpha$ such that if $f \in \Homeo_0(M)$ has $\supp(f) \subset B''_\alpha$, then $f \in W^8$.   
\end{lemma}

Similarly, for the relative case we have 

\begin{lemma}  \label{multi lemma 3}
Let $M \subset N$ be an embedded submanifold, and $\{ B_\alpha \}$ a finite collection of disjoint open balls or relative-balls in $M$.  Then there exist open balls or, respectively, relative-balls $B''_\alpha \subset B_\alpha$ such that if $f \in \Homeo_0(M)$ has $\supp(f) \subset B''_\alpha$, then $f \in W^8$.   
\end{lemma}

\noindent \textbf{Step 3: efficient covers}.  
Given lemma \ref{multi lemma 3}, the only missing ingredient to run the argument of Step 3 for manifolds with boundary or relative homeomorphism groups is Lemma \ref{cover lemma} on efficient covers.   In fact, the same proof works for this case: if $N \subset M$ is an embedded submanifold, then, \emph{provided $\epsilon$ is chosen sufficiently small}, each ball in an $\epsilon$-net cover that intersects $N$ will actually be an embedded relative ball, so the argument on existence of efficient covers runs verbatim, just replacing ``ball" by ``ball or relative ball".  Similarly, if $\del M \neq \emptyset$, any sufficiently small metric ball that intersects $\del M$ will be an embedded half-ball.

\subsection{Noncompact manifolds}

In \cite{Hurtado}, Hurtado defines a notion of \emph{weak continuity} for homomorphisms  -- a homomorphism $\phi$ from $\Homeo_0(M)$ to another group is \emph{weakly continuous} if, for every compact set $K \subset M$, the restriction of $\phi$ to the subgroup of homeomorphisms with support contained in $K$ is continuous.   
Our proof of Theorem \ref{main steinhaus thm} also shows the following. 

\begin{corollary} 
Let $M$ be any manifold, and $\phi: \Homeo_0(M) \to H$ a homomorphism to a separable topological group,.  Then $\phi$ is weakly continuous.
\end{corollary}

For general non-compact $M$, the compact-open topology on $\Homeo_0(M)$ is separable and completely metrizable.   Thus, it is reasonable to ask whether automatic continuity holds for such groups.  

\begin{question} 
Does $\Homeo_0(M)$, with the $C^0$ compact-open topology, have automatic continuity when $M$ is noncompact?  
\end{question}

\subsection{The \emph{Steinhaus} condition for Polish groups.}

In \cite{RS}, Rosendal and Solecki give a condition on a topological group that implies that the group has automatic continuity.  This condition is called \emph{Steinhaus}.

\begin{definition}
A topological group $G$ is \emph{Steinhaus} if there is some $n \in \N$ such that, whenever $W \subset G$ is a symmetric set such that countably many left-translates of $W$ cover $G$, there exists a neighborhood of the identity of $G$ contained in $W^n$.  
\end{definition}  
Note that in the definition of Steinhaus the exponent $n$ depends only on the group $G$, but the neighborhood of the identity in $G$ is allowed to depend on $W$.  
The proof that Steinhaus implies automatic continuity is a Baire category theorem argument as in Lemma \ref{Baire lemma} above.  

Our proof of automatic continuity for homeomorphism groups actually shows that $\Homeo_0(M)$ and $\Homeo_0(M \rel N)$ are Steinhaus -- the reader may check that, in each step where we referenced the set $W = \phi^{-1}(V_0)$, the only property we ever used of $W$ was that $\Homeo_0(M)$ was the union of countably many left-translates $g_i W$.  Our choice to retain the reference to $\phi^{-1}(V_0)$ was primarily for the purpose of making the proof more transparent.    In effect, what we actually proved was the following:

\begin{theorem} \label{main steinhaus thm}
Let $M$ be a compact manifold, possibly with boundary. Then $\Homeo_0(M)$ is Steinhaus. 
If $N \subset M$ is an embedded closed submanifold of dimension at least 1, then $\Homeo_0(M \rel N)$ is also Steinhaus.  
\end{theorem}


\section{Applications} \label{applications sec}

\subsection{A uniqueness result}

As a first application, we give a new short proof of Kallman's theorem from \cite{Kallman} (Theorem \ref{Kallman thm} in the introduction), that $\Homeo_0(M)$ has a unique complete, separable topology.  

\begin{proof}[Proof of Theorem \ref{Kallman thm}]
Put any complete, separable topology on $\Homeo_0(M)$, and let $H$ denote the resulting topological group.  By Theorem \ref{main cont thm}, the identity map $\Homeo_0(M) \to H$ is a continuous isomorphism of Polish groups.  Pettis' theorem (see 9.10 in \cite{Kechris}) which, in the form that we need it, is essentially a Baire category theorem arguement, now implies that this map is actually \emph{open}, hence a homeomorphism.  
\end{proof}

\subsection{Extension problems}

In \cite{EM}, Epstein and Markovic show that there is no \emph{extension} from the group of quasi-symmetric homeomorphisms of the circle to the group of quasi-conformal homeomorphisms of the disc.  In other words, there is no homomorphism $\phi: QS(S^1) \to QC(D^2)$ such that the restriction of $\phi(g)$ to the boundary of $D^2$ agrees with $g$, for each $g \in QS(S^1)$.   A major step in their proof is to show that any such map would have to be continuous.  Motivated by this, they ask whether any extension of $\Homeo_0(S^1)$ to $\Homeo_0(D^2)$ is necessarily continuous.   Theorem \ref{main cont thm} immediately gives a positive answer, as well as a positive answer to the more general problem of extensions $\Homeo(\del M) \to \Homeo(M)$.  

A less trivial application of automatic continuity is to the problem of extending \emph{germs} of homeomorphisms.  
Let $G$ be the group of orientation preserving homeomorphisms of $\R^n$ that fix the origin, and $\germ(\R^n, 0)$ the group of germs of elements of $G$ at $0$.    There is a natural (quotient) map $G \to \germ(\R^n, 0)$. Navas has asked whether this map has a group-theoretic section, in the following sense. 

\begin{question}[Navas, see also Remark 1.1.3 in \cite{DNR}]  \label{germs q}
Does there exist a homomorphism $\phi: \germ(\R^n, 0) \to G$ such that the composition $\germ(\R^n, 0) \to G \overset{\phi} \to \germ(\R^n, 0)$ is the identity? 
\end{question}

Automatic continuity implies a stronger result.  

\begin{proposition} There is no faithful homomorphism $\germ(\R^n, 0) \to \Homeo_0(\R^n)$.  In fact, there is no nontrivial homomorphism from $\germ(\R^n, 0)$ to any separable group.  
\end{proposition}

The proof only uses Rosendal and Solecki's theorem on automatic continuity of $\Homeo_0(I)$, where $I$ is a compact interval \cite{RS}, and the fact that $\germ(\R^n, 0)$ is simple, which is proved in the Appendix with Fr\'ed\'eric Le Roux.   The idea for this use of automatic continuity of $\Homeo_0(I)$ was communicated to me by C. Rosendal.  

\begin{proof}
As above, let $G$ be the group of orientation preserving homeomorphisms of $\R^n$ that fix the origin, and $I$ the interval $[0,1]$.  There is an embedding $i: \Homeo_0(I) \to G$ given by a ``radial action on the unit ball."  Precisely, put radial coordinates on the unit ball $\B^n \subset \R^n$ as $\{r \vec{v} \mid r \in [0,1], \vec{v} \in S^{n-1} \}$ and define 
$$ 
i(f)(x) = \left\{ \begin{array}{rll}  f(r) \vec{v} & \mbox{if }  x = r \vec{v} \in \B^n  \\ 
x & \mbox{otherwise} 
\end{array}\right.
$$
The image of $i(\Homeo_0(I))$ under the quotient map to $\germ(\R^n, 0)$ is abstractly isomorphic to the group of germs at $0$ of orientation-preserving homeomorphisms of $[0,1]$.  

Now suppose that $\phi: \germ(\R^n, 0) \to H$ is a homomorphism to a separable topological group.   Consider the induced homomorphism 
$$\Phi: \Homeo_0(I) \overset{i} \to G \to \germ(\R^n, 0) \overset{\phi} \to H$$
By automatic continuity for $\Homeo_0(I)$, $\Phi$ is continuous.  However, the kernel of $\Phi$ contains the subgroup of homeomorphisms that restrict to the identity in a neighborhood of $0 \in I$, and this is a dense subgroup.  It follows that $\Phi$ is trivial, and hence $\phi$ is not injective.   
Since Theorem \ref{germ thm} states that $\germ(\R^n, 0)$ is simple, $\phi$ must be trivial.  
\end{proof}

As an immediate consequence, we have the following.  
\begin{corollary} \label{germs cor}
$\germ(\R^n, 0)$ does not admit a separable group topology.  
\end{corollary}

Navas has also asked the following.

\begin{question}
Suppose that there is an isomorphism $\germ(\R^n, 0) \to \germ(\R^m, 0)$.  Is it necessarily the case that $m = n$?  
\end{question}

This question was intended to mirror the theorem of Whittaker \cite{Whittaker}, which states that if $M$ and $N$ are compact manifolds, and $\phi: \Homeo_0(M) \to \Homeo_0(N)$ is an isomorphism, then $M = N$.   Whittaker's result is essentially \emph{topological} (in fact, even without Whittaker, our Theorem \ref{main cont thm} implies that $\phi$ must be a homeomorphism!).   Corollary \ref{germs cor} implies that the corresponding question for germs of homeomorphisms is fundamentally an \emph{algebraic} question, and so likely requires completely different techniques.  
\bigskip

We conclude with a related open question that has recently attracted attention. 

\begin{question}  \label{dim q}
Suppose that $\phi: \Homeo_0(M) \to \Homeo_0(N)$ is a nontrivial (hence injective) homomorphism.  Is it necessarily true that $\dim(M) \leq \dim(N)$?    If $\dim(M) = \dim(N)$, must $\phi$ come from an embedding or a covering map?  
\end{question}

In \cite{Hurtado}, an analogous result is proved for groups of smooth diffeomorphisms of manifolds.  The first step in the proof is to show that such a homomorphism is necessarily continuous.  Our theorem \ref{main cont thm} gives this in the case of $\Homeo_0(M)$; this should represent significant progress towards the solution of Question \ref{dim q}.

\subsection{Algebraic nonsmoothing}  \label{nonsmoothing sec}

Recall that in the introduction we defined the \emph{regularity} of an abstract group $G$ to be the largest $r$ such that there exists a manifold $M$ and a nontrivial homomorphism $G \to \Diff^r(M)$.   The following question appears to be wide open. 

\begin{question} \label{reg q}
Give examples of groups of regularity $r$, for any given $r$.  Are there examples which are finitely or compactly generated groups?  
\end{question}

In \cite{Navas}, Navas discusses the related question of finitely generated subgroups of $\Homeo_0(S^1)$ that do not act by $C^1$ diffeomorphisms on $S^1$.  However, his work relies heavily on the 1-dimensional setting, and it is conceivable that his examples could act by diffeomorphisms on a manifold of higher dimension, hence still be \emph{algebraically $C^1$--smoothable}, in the sense defined in the introduction.  

We give the first partial answer to Question \ref{reg q} now, showing that $\Homeo_0(M)$ has regularity 0.  

\begin{proof}
Suppose for contradiction that $\phi: \Homeo_0(M) \to \Diff^1(N)$ were such a homomorphism.  The topology on $\Diff^1(N)$ is separable, it is induced by the $C^1$ norm 
$$\| f\|_1 = \sup_{x,y \in M} \{ d_N(f(x), x) + \|Df(y)\| \}$$ 
where $d_N$ is any compatible metric on $N$.  Here one needs to define a suitable notion of norm of derivative.  For the purposes of this proof, we take this norm to be defined by fixing an embedding $N \subset \R^m$ for some $m$, and taking the operator norm of $Df(y)$ as a map between the tangent spaces at $y$ and $f(y)$.   We take the metric on $N$ to be the path metric induced by the embedding in $\R^m$.

Fix $\epsilon > 0$.  Since $\phi$ is continuous by Theorem \ref{main cont thm}, there exists a neighborhood $U$ of the identity in $\Homeo_0(M)$ such that $\|\phi(f) \|_1 \leq \epsilon$ for all $f \in U$.   In particular, if $B \subset M$ is an embedded ball of sufficiently small diameter, and $G_B$ the group of homeomorphisms of $M$ supported on $B$, then $\phi(G_B) \subset U$.  Let $B$ be such a ball, and $B' \subset B$ a smaller ball with closure contained in $B$.  

Let $g \in \Homeo_0(M)$ be a contraction of $B'$ supported on $B$.  By this, we mean that 
\begin{enumerate}[i)]
\item $g(B') \subset B'$
\item $\bigcap \limits_{n=1}^\infty g^n(B') = \{p\}$ for some point $p \in B'$
\item $\supp(g) \subset B$.  
\end{enumerate}
Note also that $\supp(g^n) \subset B$, so $g^n \in U$, and hence $\|\phi(g^n) \|_1 \leq \epsilon$, for all $n \in \Z$.  

Let $h$ be supported in $G_B'$.  Since $\Homeo_0(M)$ is simple, $\phi(h)$ is nontrivial, so there exists $x_0 \in N$ with $d(\phi(h)(x_0), x_0) =\delta > 0$.  
Since $g$ contracts $B'$, as $n \to \infty$, $g^n h g^{-n} \to \id$ in $\Homeo_0(M)$.  By continuity, $\phi(g^n) \phi(h) \phi(g^{-n}) \to \id$ in $\Homeo(N)$.   In particular, if $n$ is large enough,  then
$$\sup_{y \in N} d_N \left( \phi(g^n h g^{-n})(y), y \right) < \delta/\epsilon.$$

In particular, taking $y = \phi(g^n)(x_0)$, this means that  $d_N \left( \phi(g^n h)(x_0), \phi(g^n)(x_0) \right) < \delta/\epsilon$.  Consider a geodesic segment $\gamma$ on $N$ from $\phi(g^n h)(x_0)$ to $\phi(g^n)(x_0)$.  Then $\phi(g^{-n})(\gamma)$ is a $C^1$ path from $\phi(h)(x_0)$ to $x_0$, so has length greater than $\delta$.   It follows that $\phi(g^{-n})$ expands the length of a differentiable path by a factor of more than $\epsilon$.  But this contradicts the fact that $\phi(g^{-n}) \in U$, so $\sup_{x \in M} \|Dg^n(x) \| < \epsilon$. 
 
\end{proof}

We conjecture that an analogous result holds for diffeomorphism groups.     
\begin{conjecture}
The group $\Diff^r(M)$ has regularity $r$.   
\end{conjecture} 

A good first step would be to prove automatic continuity for such groups. 
\begin{question}
Does $\Diff^r_0(M)$ have the automatic continuity property?   If so, is $\Diff^r_0(M)$ Steinhaus?
\end{question}
\vspace{.5cm}

\newpage

\appendix
\section{Appendix: Structure of groups of germs}  \label{appendix}
\begin{center}
\author{Fr\'ed\'eric Le Roux
and Kathryn Mann
}
\end{center}

\begin{theorem} \label{germ thm}
$\germ(\R^n, 0)$ is \emph{uniformly simple} in the following strong sense: given any nontrivial $g \in \germ(\R^n, 0)$, every  element $g' \in \germ(\R^n, 0)$ can be written as a product of 8 conjugates of $g$.
\end{theorem} 

The argument that we give here applies to the case $n \geq 2$.  A short argument for simplicity of $\germ(\R, 0)$ can be found in Proposition 4 of \cite{Mann LO} (using a very similar strategy of proof to the one here). 

\begin{no/co} 
Let $G$ denote the group of orientation preserving homeomorphisms of $\R^n$ fixing $0$.   Recall that $\germ(\R^n, 0)$ is the quotient of $G$ by the subgroup of homeomorphisms that restrict to the identity in a neighborhood of $0$.    
By convention, a \emph{ball containing 0} is an image $B$ of the standard closed unit ball under a global homeomorphism of $\R^n$,  with $0$ in the interior of $B$.  We use the symbol $\mathring{A}$ to denote the interior of a set $A$.  
\end{no/co} 

\begin{definition}
An element $g \in G$ is a \emph{local contraction} if there exists a ball $B$ containing $0$ such that $g(B) \subset \mathring{B}$ and such that $\bigcap \limits_n g^n(B) = \{0\}$ 
\end{definition}

\begin{lemma} \label{contraction lem}  
The germs of any two local contractions are conjugate.  
\end{lemma}

\begin{proof}
We show that any local contraction has germ conjugate to that of $x \mapsto  \frac{1}{2}x$.  
The proof uses the annulus theorem.  

Let $g$ be a local contraction.  After conjugacy, we may assume that the ball $B$ contracted by $g$ is the standard unit ball.  By the annulus theorem, there exists a homeomorphism $h_1: B \setminus g(\mathring{B}) \to B \setminus  \frac{1}{2} \mathring{B}$ that is the identity on $\del B$, and, inductively $h_n: \left( g^{n-1}(B) \setminus g^n(\mathring{B} \right) \to \left( 2^{-n+1}B \setminus 2^{-n} \mathring{B} \right)$ agreeing with $h_{n-1}$ on $\del (g^{n-1}(B))$.   Define a homeomorphism
$$h(x) = \left\{ \begin{array}{ll} 
h_n(x) & \text{ if } x \in g^{n-1}(B) \setminus g^n(\mathring{B}) \\
x & \text{ otherwise.}
\end{array} \right. $$
Then $hgh^{-1}(2^{-n+1}B) = 2^{-n}B$.  Let $\hat{g} = hgh^{-1}$.  Now we build another conjugacy to ``straighten" $\hat{g}$ to the standard contraction $x \mapsto \frac{1}{2}x$.  

The restriction of $\hat{g}$ to $\del B$, considered as a homeomorphism $\del B = S^{n-1} \to \del(  \frac{1}{2}(B)) = S^{n-1}$, is isotopic to the identity (this is a consequence of Kirby's stable homeomorphism theorem); let $g_t$, $t \in [1/2, 1]$ be such an isotopy with $g_1 = \hat{g}$ and $g_{1/2} = \id$.  Identify $B$ with $\{r s \mid r \in [0,1], s \in S^{n-1} \}$, and define a foliation of $B \setminus  \frac{1}{2}B$, transverse to the boundary, with 1-dimensional leaves of the form
$$L_s := \{r g_r(s) \mid r \in [1/2, 1]\}. $$
This extends naturally to a leafwise $\hat{g}$-invariant foliation on $B \setminus 0$ with leaves equal to $\bigcup_{n\geq 0} \hat{g}^n(L_s)$; we will produce a conjugacy that straightens these to radii of $B$.   Note that each $x \in B \setminus \{0\}$ can be written uniquely as $\hat{g}^n(rg_r(s))$ for some $n \in \N$, $r\in (1/2,1]$, and $s \in S^{n-1}$.  
Define $\hat{h}$ by 
$$\hat{h}(x) = \left\{ \begin{array}{ll} 
2^{-n}r s & \text{ if } x = \hat{g}^n(r g_r(s) \\
x & \text{ otherwise.}
\end{array} \right. $$
Then $\hat{h} \hat{g} \hat{h}^{-1}$ preserves each radius of $B$, and is conjugate to $x \mapsto  \frac{1}{2} x$ on each radius.  A continuous choice of conjugacies on radii gives a conjugacy of $\hat{g}$ with a homeomorphism agreeing with $x \mapsto \frac{1}{2}x$ on $B$.  
\end{proof}

\begin{definition}  
An element $g \in G$ \emph{contracts a basis of balls} if there exist nested balls $B_1 \supset B_2 \supset ... $ containing 0 with $\bigcap_n B_n = \{0\}$, and such that $g(B_n) \subset \mathring{B}_n$ for all $n$.  
\end{definition}

\begin{lemma} \label{basis lem1}
Let $g \in G$ have nontrivial germ at 0.  Then there exists $a \in G$ such that $aga^{-1}g$ contracts a basis of balls. 
\end{lemma}

\begin{proof}
In this proof, we let $B(r, x)$ denote the ball of radius $r$ about $x$.  

Let $g$ have nontrivial germ at 0.  Then in any neighborhood $U_1$ of 0, there exists $x$ such that $g(x) \neq x$.  This means we can take a ball $B$ containing 0 and $x$, but not $g^{-1}(x)$, in particular $x \notin gB$.   Then we can find a ball $B'$ with $g(B) \subset \mathring{B}'$ and such that the pair $B, B'$ is homeomorphic to the pair $B(2, -z), B(2, z)$, where $z = (1, 0, 0, ...) \in \R^n$ (see Figure~\ref{fig1}).     We may also take $B'$ and $B$ to be contained in $U_1$.  

\begin{figure}[h]
\centerline{
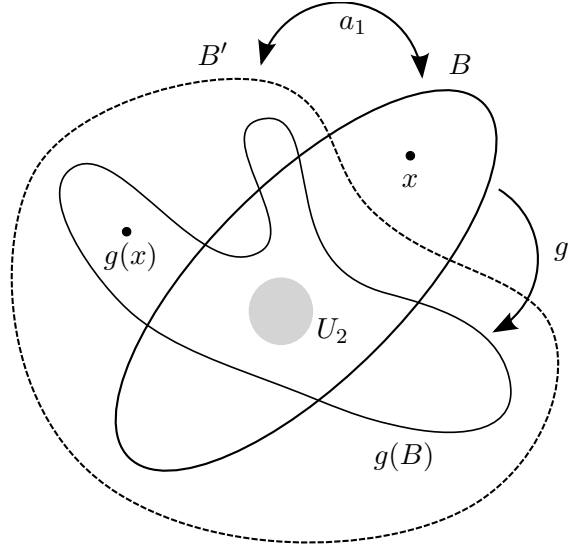
}
\caption{\label{fig1} The balls $B$, $g(B)$, and $B'$}
\end{figure}

Let $h: \R^n \to \R^n$ be a homeomorphism with $h(B) = B(2, -z)$ and $h(B') = B(2,z)$.  
There is a homeomorphism $r$, supported on a small neighborhood of $B(2, -z) \cup B(2, z)$, with $r(B(2,z)) = B(2, -z)$ and $r(B(2,-z)) = B(2, z)$, and such that $r$ fixes pointwise a small ball around $0$.  So $hrh^{-1}$ exchanges $B$ and $B'$, and $hrh^{-1}$ is the identity on a ball $U_2$ containing 0.  We may also take $hrh^{-1}$ to be supported on $U_1$.  Let $a_1 = hrh^{-1}$.  Note that since $g B \subset \mathring{B'}$ and $a_{1}$ exchanges $B$ and $B'$  we get $a_{1}g a_{1}^{-1} B' \subset \mathring{B}$.
Then $a_1 g a_1^{-1} g( B) \subset a_1 g a_1^{-1}\mathring{B'}  \subset \mathring{B}$.  

Repeating the construction above, using $U_2$ in place of $U_1$, we can find $h_2$ supported on $U_2$, fixing a smaller neighborhood $U_3$ of 0, and so that $h_2 g h_2^{-1} g (B_2) \subset \mathring{B_2}$ for some $B_2 \subset U_2$.   In the same manner, inductively define $a_n$, with $\supp(a_n) \subset U_n \setminus U_{n-1}$, and such that $a_n g a_n^{-1} g (B_n) \subset \mathring{B_n}$ for some ball $B_n \subset U_n$ containing 0.  We may also choose $B_n$ so that $\bigcap B_n = \{0\}$ (e.g. at each stage, ensure that $B_n$ is contained in a ball of radius $2^{-n}$).  

Define a homeomorphism 
$$a(x) =  \left\{ \begin{array}{ll} 
a_n(x) & \text{ if } x \in U_n \setminus U_{n-1} \\
x & \text{ otherwise}
\end{array} \right. $$
Then $hgh^{-1}(B_n) \subset \mathring(B_n)$ for all $n$.  

\end{proof}

\begin{lemma}  \label{basis lem2}
Let $f \in G$.  If $f$ contracts a basis of balls, then there exists $b$ such that $bfb^{-1}f$ is a local contraction.  
\end{lemma} 

\begin{proof}
We return to the notation of Lemma \ref{contraction lem}, where $r B$ denotes the ball of radius $r$ centered at 0   
Using the annulus theorem as in the first part of the proof of Lemma \ref{contraction lem}, we may conjugate $f$ to a homeomorphism that contracts the nested balls $B \supset \frac{1}{2}B \supset \frac{1}{4}B \supset ...$
Abusing notation, let $f$ denote this new map.   Choose $r_n \in (2^{-n-1}, 2^{-n})$ so that $f(2^{-n} B) \subset r_n B $.  Let $\lambda: [0, \infty) \to [0, \infty)$ be a homeomorphism such that $\lambda(2^{-n}) = r_n$, and $\lambda(r_n) = 2^{-n-1}$.  Let $b: \R^n \to \R^n$ be defined by 
$$b(x) = \lambda(\|x\|) \frac{x}{\|x\|}.$$
Then 
$$b f b^{-1} f ( 2^{-n}B) \subset b f b^{-1} r_n B =  bf 2^{-n}B \subset b r_n B = 2^{-n-1} B$$ 
so $b f b^{-1} f$ is a local contraction.  
\end{proof}

Combining Lemmas \ref{contraction lem}, \ref{basis lem1} and \ref{basis lem2} immediately gives the following.  

\begin{corollary}  \label{key cor} 
Let $g \in G$ have nontrivial germ.  Then any local contraction can be written as the product of 4 conjugates of $g$.  
\end{corollary}

We can now easily finish the proof of Theorem \ref{germ thm}.  

\begin{proof}[Proof of Theorem \ref{germ thm}]
Let $g$ and $g'$ be elements of $G$, and assume $g$ has nontrivial germ.  We first construct a local contraction $c$ such that $cg'$ is also a local contraction.   Let $r_n = \max \{ \|x\| : x \in g'(2^{-n}B) \}$, and let $c$ be a local contraction mapping $r_n B$ to  $t_n B$ where $t_n < \min \{r_n, 2^{-n-1} \}$ -- it is easy to construct such a map that preserves each ray through 0.    Then $cg'(2^{-n}B) \subset 2^{-n-1}(B)$, so $cg'$ is a local contraction.  

By Corollary \ref{key cor}, $cg'$ can be written as a product of 4 conjugates of $g$.  Since $g^{-1}$ also has nontrivial germ, Corollary \ref{key cor} implies that $c$ can be written as a product of 4 conjugates of $g^{-1}$.  Thus, $g' = c^{-1}cg'$ can be written as a product of 8 conjugates of $g$.  

\end{proof}


\vspace{.3in}

Dept. of Mathematics 

University of California, Berkeley  

970 Evans Hall

Berkeley, CA 94720 

E-mail: kpmann@math.berkeley.edu

\vspace{.3in}

\textit{Fr\'ed\'eric Le Roux:}

Institut de maths de Jussieu

4 place Jussieu, Case 247, 75252 Paris C\'edex 5

e-mail: frederic.le-roux@imj-prg.fr

\end{document}